\documentclass[a4paper,11pt,dvipdfmx]{article}
\usepackage{amsfonts}
\usepackage{amssymb}
\usepackage{amsmath}
\usepackage{amsthm}
\usepackage{bm}
\usepackage{here}
\usepackage{tikz}
\usetikzlibrary{intersections,calc,arrows}
\usepackage{color}
\usepackage{cases}
\numberwithin{equation}{section}

\theoremstyle{definition}
\newtheorem{defn}{Definition}[section]
\newtheorem{rem}[defn]{Remark}
\theoremstyle{plain}

\newtheorem{prop}[defn]{Proposition}
\newtheorem{thm}[defn]{Theorem}
\newtheorem{cor}[defn]{Corollary}
\newtheorem*{mainthm}{Main Theorem}
\title{Connection formula for the Jackson integral of Riemann-Papperitz type
}
\author{
	Taikei Fujii\footnote{Department of Mathematics, Graduate School of Science, Kobe University, Rokko, Kobe 657-8501, Japan.
		E-mail:{ tfujii@math.kobe-u.ac.jp}}\ \ 
	and  
	Takahiko Nobukawa
	\footnote{Department of Mathematics, Graduate School of Science, Kobe University, Rokko, Kobe 657-8501, Japan.
		E-mail: tnobukw@math.kobe-u.ac.jp\\\ \\
		keywords: \\
		Jackson integral, $q$-hypergeometric equation, 
		connection problem, very-well-poised $q$-hypergeometric series.}
}
\date{}
\allowdisplaybreaks
\begin{document}
	\maketitle
	\begin{abstract}
		We give a connection formula for the Jackson integral of Riemann-Papperitz type.
		This includes a solution of the connection problem for the variant of $q$-hypergeometric equation of degree three introduced by Hatano-Matsunawa-Sato-Takemura.
		Using this formula we show a linear relation for the Kajihara's  $q$-hypergeometric series $W^{M,2}$, whose $q$-difference equation is recently obtained by one of the author. 
	\end{abstract}
	\section{Introduction}
The Heine's $q$-hypergeometric equation  
\begin{align}\label{Int: qHGEq}
	[x(1-aT_{x})(1-bT_{x})-(1-T_{x})(1-cq^{-1}T_{x})]f(x)=0,
\end{align}
has the integral solutions
\begin{align}\label{Int: HISol}
&\int_{\tau_1}^{\tau_2}t^{\alpha - 1}\frac{(qt, b x t)_{\infty}}{(c t/a, x t)_{\infty}}d_qt, \quad \tau_1, \tau_2 \in \{0, 1, q/b x \}, \\
\label{Int: HISol2}
&x^{-\beta}\int_{\sigma_1}^{\sigma_2}s^{\beta  - \gamma }\frac{(qas/c, qs/x)_{\infty}}{(s, q s/b x )_{\infty}}d_qs, \quad \sigma_1,  \sigma_2 \in \{0, x, c/a\}.
\end{align}
where $a = q^{\alpha}, b = q^{\beta}, c = q^{\gamma}$ and $T_xf(x) = f(x)$.
The other notations in \eqref{Int: HISol}, \eqref{Int: HISol2} are given in section \ref{secpre}.
The integral solutions \eqref{Int: HISol}, \eqref{Int: HISol2} satisfy the following linear relations (see \cite{Mi1989}, for example): 
\begin{align}\label{Int: qCPIS}
\int_0^1t^{\alpha - 1}\frac{(qt, bxt)_{\infty}}{(ct/a, xt)_{\infty}}d_qt
&=
\left(\frac{b}{q}\right)^{\alpha}
\frac{\theta(b)}{\theta(b/a)}x^{\alpha}\frac{\theta(ax)}{\theta(x)}
\times
\int_0^{q/bx}t^{\alpha - 1}\frac{(qt, bxt)_{\infty}}{(ct/a, xt)_{\infty}}d_qt
\nonumber
\\
&+
\left(\frac{a}{c}\right)^{\beta - \gamma +1}
\frac{\theta(c/b)}{\theta(a/b)}x^{\beta}\frac{\theta(bx)}{\theta(x) }
\times
x^{-\beta}
\int_0^{c/a}s^{\beta- \gamma}\frac{(qas/c, qs/x)_{\infty}}{(s, q s/b x )_{\infty}}d_qs.
\end{align}
Such a connection problem is an important and fundamental problem for the theory of hypergeometric functions.
There are many studies on connection problems for various $q$-difference equations, and the works relevant for us are \cite{It,IN,Mi1989,N,Wat}.

In \cite{HMST} , an extension of the Heine's equation was introduced as follows: 
		\begin{align}\label{Int: TakemuraH3}
			\mathcal{H}_{3}&y=0,\\
			\notag \mathcal{H}_{3}&=\prod_{i=1}^{3}(x-q^{h_{i}+1/2}t_{i})\cdot T_{x}^{-1}+q^{2\alpha+1}\prod_{i=1}^{3}(x-q^{l_{i}-1/2}t_{i})\cdot T_{x}\\
			\notag &+q^{\alpha}\biggl[-(q+1)x^{3}+q^{1/2}\sum_{i=1}^{3}(q^{h_{i}}+q^{l_{i}})t_{i}x^{2}\\
			\notag &\phantom{+q^{\alpha}}-q^{(h_{1}+h_{2}+h_{3}+l_{1}+l_{2}+l_{3}+1)/2}t_{1}t_{2}t_{3}\sum_{i=1}^{3}((q^{-h_{i}}+q^{-l_{i}})/t_{i})x\\
			&\phantom{+q^{\alpha}}+q^{(h_{1}+h_{2}+h_{3}+l_{1}+l_{2}+l_{3})/2}(q+1)t_{1}t_{2}t_{3}\biggr].
		\end{align}
The equation (\ref{Int: TakemuraH3}) is called the variant of  $q$-hypergeometric equation of degree three.
It is shown in  \cite{AT,FN} that the equation  (\ref{Int: TakemuraH3}) has the integral solution of the form
\begin{align}\label{integrandJP3}
\int_{\tau_1}^{\tau_2}\frac{(Axt)_{\infty}}{(Bxt)_{\infty}}\prod_{i=1}^3\frac{(a_i t)_{\infty}}{(b_i t)_{\infty}}d_qt, \quad \tau_1, \tau_2\in\{q/a_1,q/a_2,q/a_3,q/(Ax)\},
\end{align}
where
\begin{align}
A a_1 a_2 a_3 =q^2 B b_1 b_2 b_3.
\end{align}
The integral \eqref{integrandJP3} is a $q$-analog of the following integral: 
\begin{align}\label{Int: RPI}
\int_C (1-t_1 t)^{\nu_1}(1-t_2 t)^{\nu_2}(1-t_3 t)^{\nu_3}(1-xt)^{\nu_4}dt,\quad \nu_1+\nu_2+\nu_3+\nu_4=-2.
\end{align}
The Riemann-Papperitz's  equation is a Fuchsian differential equation of second order with three singularities $\{ t_1, t_2, t_3\}$ on the Riemann sphere $\mathbb{C}\cup \{\infty \}$, and has integral solution of the form \eqref{Int: RPI}.
So we call the integral \eqref{integrandJP3} the Jackson integral of Riemann-Papperitz type in this paper. 
For the details of the Riemann-Papperitz's differential equation, see \cite{WW}.

In this paper, generalizing the integral \eqref{integrandJP3},  we will consider connection formulas for the following integrals for $M\geq 1$:
\begin{align}\label{integrandJP}
\int_{q/a_i}^{q/a_j}\prod_{k=1}^{M+3}\frac{(a_k t)_{\infty}}{(b_k t)_{\infty}}d_{q}t, \quad a_1 a_2 \cdots a_{M+3} =q^2 b_1 b_2 \cdots b_{M+3}.
\end{align}
We also call this integral the Jackson integral of Riemann-Papperitz type.
When $M = 1$, these formulas give a solution of the connection problem for the  equation \eqref{Int: TakemuraH3}.
Our main result is as follows:
\begin{mainthm}\label{Int: thmconnint}
Suppose $a_1\cdots a_{M+3}=q^2 b_1 \cdots b_{M+3}$.
We have
\begin{align}\label{Int: connvarM}
	\sum_{k=2}^{M+3} \tilde{C}_k \int_{q/a_1}^{q/a_k} \prod_{i=1}^{M+3}\frac{(a_i t)_\infty}{(b_i t)_\infty} d_q t=0.
\end{align}
Here 
\begin{align}
	\tilde{C}_k=\left(\frac{a_k}{a_1}\right)^2\prod_{i=1}^{M+3}\frac{\theta(a_k/b_i)}{\theta(a_1/b_i)}\prod_{\substack{1\leq i\leq M+3\\i\neq 1}}\theta(a_1/a_i)\prod_{\substack{1\leq i\leq M+3\\i\neq k}}\theta(a_k/a_i)^{-1}.
\end{align}
\end{mainthm}

Recently it is shown that the integral \eqref{integrandJP} is expressed by the Kajihara's $q$-hypergeometric series \cite{Nkaji}.
Therefore, our main result \eqref{Int: connvarM} gives a new linear relation for the Kajihara's $q$-hypergeometric series.
For the background and details for this series, see \cite{Kaji,kaji2018,KN}.

The contents of this paper are as follows. 
In section \ref{secpre}, we define the notations.
In section \ref{Sec con}, we derive a connection formula among the Jackson integrals of Riemann-Papperitz type \eqref{integrandJP}. 
In section \ref{Secsumdis}, we summarize  the results and discuss  related problems.  

\section{Preliminaries}\label{secpre}
	Throughout this paper, we fix $q\in\mathbb{C}$ with $0<|q|<1$, and use the following notations:
	\begin{align}
		&(a)_\infty=\prod_{i=0}^\infty(1-aq^i),\ (a)_l=\frac{(a)_\infty}{(aq^l)_\infty},\ (a_1,\ldots,a_r)_l=(a_1)_l\cdots (a_r)_l,\\
		&\theta(x)=(x,q/x)_\infty,\ \theta(x_1,\ldots,x_r)=\theta(x_1)\cdots \theta(x_r),\\
		&\int_0^\tau f(t)d_qt=(1-q)\sum_{n=0}^\infty f(\tau q^n)\tau q^n,\ 
		\int_0^{\tau\infty}f(t)d_qt=(1-q)\sum_{n=-\infty}^\infty f(\tau q^n)\tau q^n,\\
		&\int_{\tau_1}^{\tau_2}f(t)d_qt=\int_{0}^{\tau_2}f(t)d_qt-\int_{0}^{\tau_1}f(t)d_qt,\\
		&T_x f(x)=f(qx).
	\end{align}
	In the following, we suppose that the parameters satisfy 
	\begin{align}
		Aa_2a_3\cdots a_{M+3}=q^2 Bb_2b_3\cdots b_{M+3}.
	\end{align}
		For $M\in\mathbb{Z}_{>0}$, the following $q$-difference operator $E_M$ of rank $M+1$ is defined in \cite{Nkaji} as follows:
		\begin{align}
			\notag E_M&=x^{M+2} T_x^{-1}\prod_{i=0}^{M}(B-Aq^iT_x)\\
			\notag&+\sum_{k=1}^{M+1}(-1)^{k}x^{M+2-k}[e_{k}(a)T_{x}^{-1}-qe_{k}(b)]\prod_{i=0}^{M-k}(B-Aq^iT_{x})\prod_{i=0}^{k-2}(1-q^{-i}T_{x})\\
			\label{defequationEM}&+(-1)^M \frac{a_2\cdots a_{M+3}}{B}T_x^{-1}\prod_{i=0}^{M}(1-q^{-i}T_x).
		\end{align}
		Here, $a=(a_2,a_3,\ldots,a_{M+3})$, $b=(b_2,b_3,\ldots,b_{M+3})$ and $e_k$ is the elementary symmetric polynomial of degree $k$.
The following integral solution for $E_My=0$ is obtained in \cite{Nkaji}.
	\begin{thm}[\cite{Nkaji}]\label{propequationint}
		If $\tau\in\{q/(Ax),q/a_2,q/a_3,\ldots,q/a_{M+3}\}$, then the integral
		\begin{align}
			\int_0^\tau \frac{(Axt)_\infty}{(Bxt)_\infty}\prod_{k=2}^{M+3}\frac{(a_kt)_\infty}{(b_kt)_\infty}d_qt,
		\end{align}
		satisfies the non-homogeneous equation
		\begin{align}\label{E_My=0}
			E_M y=-\prod_{i=0}^{M-1}(B-Aq^i)\cdot q(1-q)x^{M+1}.
		\end{align}
		In particular, the Jackson integral of Riemann-Papperitz type
		\begin{align}\label{intRP}
			\varphi_{i,j}=\int_{q/a_i}^{q/a_j} \frac{(Axt)_\infty}{(Bxt)_\infty}\prod_{k=2}^{M+3}\frac{(a_kt)_\infty}{(b_kt)_\infty}d_qt,
		\end{align}
		satisfies the equation $E_My=0$.
		Here, $a_1=Ax$.
	\end{thm}
	\begin{rem}\label{remparameterstrans}
		When $M=1$, the equation $E_1y=0$ is equivalent to the  equation \eqref{Int: TakemuraH3}.
		More precisely, by putting $ q^{-\nu}=B/A$, $q^{l_{i}-1/2}t_{i}=b_{i+1}/A$, $q^{h_{i}+1/2}t_{i}=a_{i+1}/B$, and by the gauge transformation $Y=x^{\nu-\alpha}y$, we have $\mathcal{H}_3Y=0$ if $E_1y=0$.
		And we have
		\begin{align}\label{valH3t}
			&\mathcal{H}_3\varphi_{i,j}^{\mathcal{H}_3}=0,\\
			&\varphi_{i,j}^{\mathcal{H}_3}=x^{\nu -\alpha } \int_{\tau_i}^{\tau_j} \frac{(q^{\nu } x t, q^{h_1+\frac{1}{2}} 
				t_1 t, q^{h_2+\frac{1}{2}} t_2 t, q^{h_3+\frac{1}{2}}
				t_3 t)_{\infty} }{(x t,q^{\nu +l_1-\frac{1}{2}} 
				t_1 t,q^{\nu +l_2-\frac{1}{2}} t_2 t,q^{\nu +l_3-\frac{1}{2}} 
				t_3 t)_{\infty} }  d_q t,\label{intH3}
		\end{align}
		where $\tau_i=q^{-h_i-1/2}/t_i$ $(i=1,2,3)$, $\tau_4=q^{1-\nu}/x$.
		Proposition \ref{propequationint} for the case of $M=1$ was obtained in \cite{AT, FN}.
	\end{rem}
%
	
	\section{Connection problem}\label{Sec con}
	In this section, we give a connection formula among the Jackson integrals of Riemann-Papperitz type \eqref{intRP}.
	Though this formula can be obtained by reducing Mimachi's formula \cite{Mi1989}, we will give independent alternative proof (see Remark \ref{thetarem} below).

	A connection formula for the Jackson integral of Jordan-Pochhammer type was obtained by Mimachi \cite{Mi1989} by calculating a certain contour integral.
	First, we review Mimachi's connection formula \cite{Mi1989} for the Jackson integral of Jordan-Pochhammer type.
	See also \cite{It,It2023,Ma,Mi1994} for the Jackson integral of Jordan-Pochhammer type and its extensions.

\begin{prop}[\cite{Mi1989}]\label{propconnMimachi}
We have
\begin{align}\label{connJP}
	{\theta(q^\rho)}\int_0^{qb_{1}} t^{\rho-1} \prod_{i=1}^{M+3}\frac{(t/b_i)_\infty}{(t/a_i)_\infty}d_q t=\sum_{k=1}^{M+3} C_k \int_0^{q/a_k} t^{\alpha-1} \prod_{i=1}^{M+3}\frac{(a_i t)_\infty}{(b_i t)_\infty} d_q t,
\end{align}
where $q^\rho=a_1\cdots a_{M+3}/(b_1\cdots b_{M+3}q^\alpha)$ and
\begin{align}
	C_k=(qb_{1})^\rho\left(\frac{a_k}{q}\right)^\alpha{\theta(q^{\rho+1}b_1/a_k)}\prod_{j=2}^{M+3} \theta\left(\frac{a_k}{b_j}\right)\prod_{\substack{1\leq j\leq M+3\\j\neq k}}\theta\left(\frac{a_k}{a_j}\right)^{-1}.
\end{align}
\end{prop}
\begin{proof}
We consider the following contour integral:
\begin{align}
	\int_L \frac{(c_1 s,\ldots,c_{M+2} s, qs/x, x / s  )_\infty}{(d_1 s,\ldots, d_{M+3} s, 1 / s)_\infty}\frac{ds}{s}.
\end{align}
Here the contour $L$ is a deformation of the positively oriented unit circle so that the poles of $1/(d_1 s,\ldots, d_{M+3} s)_\infty$ lie outside $L$, and the poles of  $1/(1 / s)_\infty$ and $0$ lie inside $L$.
By the Cauchy's residue formula to the inside and outside, we get the following formula:
\begin{align}
	\notag&{\theta(x)}\sum_{n=0}^\infty\frac{(qq^n,c_1q^n,\ldots,c_{M+2}q^n)_\infty}{(d_1 q^n,\ldots,d_{M+3}q^n)_\infty}x^n\\
	&=\sum_{k=1}^{M+3} {\theta(xd_k)}\prod_{i=1}^{M+2}\theta(c_i/d_k)\prod_{\substack{1\leq i\leq M+3\\i\neq k}}\theta(d_i/d_k)^{-1}\sum_{n=0}^\infty\frac{(qd_kq^n/d_1,\ldots,qd_kq^n/d_{M+3})_\infty}{(qd_kq^n/c_1,\ldots,qd_kq^n/c_{M+2},d_kq^n)_\infty}\left(\frac{c_1\cdots c_{M+2} q}{d_1\cdots d_{M+3}x}\right)^n.
\end{align}
Putting $c_i$, $d_i$ and $x$ suitably, we obtain the desired formula (\ref{connJP}).
\end{proof}
\begin{rem}
The formula (\ref{connJP}) is equivalent to the connection formula of the generalized $q$-hypergeometric series ${}_{M+3}\varphi_{M+2}$ \cite{Tho,Wat}.
\end{rem}
\begin{rem}
A connection formula for the Jackson integral of type $A$ was studied by Ito-Noumi \cite{IN}.
Their formula includes Mimachi's one (\ref{connJP}) as a special case.
\end{rem}
We put $\alpha=1$ and $a_1 \cdots a_{M+3}=q^2 b_1 \cdots b_{M+3}$, then $\theta(q^\rho)=\theta(q)=0$.
So we get
\begin{align}
0=\sum_{k=1}^{M+3} C_k \int_0^{q/a_k}  \prod_{i=1}^{M+3}\frac{(a_i t)_\infty}{(b_i t)_\infty} d_q t.
\end{align}
By a simple calculation, we have
\begin{align}
\sum_{k=2}^{M+3} \frac{C_k}{C_{1}} \int_{q/a_{1}}^{q/a_k} \prod_{i=1}^{M+3}\frac{(a_i t)_\infty}{(b_i t)_\infty} d_q t=-\left(1+\frac{C_2}{C_1}+\cdots+\frac{C_{M+3}}{C_1}\right)\int_{0}^{q/a_1}\prod_{i=1}^{M+3}\frac{(a_i t)_\infty}{(b_i t)_\infty} d_q t.
\end{align}
We put $a_1=A x$ and $b_1=B x$.
Due to Proposition \ref{propequationint}, the left-hand-side of the above formula satisfies the equation $E_M y=0$ because coefficients $C_k/C_1$ ($k=2,\ldots,M+3$) are pseudo constants of $x$.
On the other hand, the integral 
\begin{align}\label{intnotsol}
\int_{0}^{q/a_1}\prod_{i=1}^{M+3}\frac{(a_i t)_\infty}{(b_i t)_\infty} d_q t,
\end{align}
of the right-hand-side is not a solution for $E_M y=0$.
Thus the coefficient $1+C_2/C_1+\cdots+C_{M+3}/C_1$ vanishes.
Therefore we have the following formula:
\begin{thm}\label{thmconnint}
Suppose $a_1\cdots a_{M+3}=q^2 b_1 \cdots b_{M+3}$.
We have
\begin{align}\label{connvarM}
	\sum_{k=2}^{M+3} \tilde{C}_k \int_{q/a_1}^{q/a_k} \prod_{i=1}^{M+3}\frac{(a_i t)_\infty}{(b_i t)_\infty} d_q t=0.
\end{align}
Here 
\begin{align}
	\tilde{C}_k=\frac{C_k}{C_1}=\left(\frac{a_k}{a_1}\right)^2\prod_{i=1}^{M+3}\frac{\theta(a_k/b_i)}{\theta(a_1/b_i)}\prod_{\substack{1\leq i\leq M+3\\i\neq 1}}\theta(a_1/a_i)\prod_{\substack{1\leq i\leq M+3\\i\neq k}}\theta(a_k/a_i)^{-1}.
\end{align}
\end{thm}
\begin{rem}\label{thetarem}
We obtained the linear relation (\ref{connvarM}) by reducing Mimachi's formula (\ref{connJP}).
The key point of our proof is to show the relation $1+C_2/C_1+\cdots+C_{M+3}/C_1=0$ by a simple argument using the property of the integral (\ref{intnotsol}).
{Note that the formula $1+C_2/C_1+\cdots+C_{M+3}/C_1=0$ was discussed in \cite{TM}.
See also \cite[p. 474]{WW}.
}
\end{rem}
\begin{rem}
	In \cite{Nkaji}, $\displaystyle\binom{M+3}{2}$ integral solutions $\varphi_{i,j}$ \eqref{intRP} for the rank $M+1$ equation $E_My=0$ were given (see also Proposition \ref{propequationint}).
	Using the connection formula \eqref{connvarM} and the trivial relation $\varphi_{i,j}+\varphi_{j,k}+\varphi_{k,l}$, every integral $\varphi_{i,j}$ can be written by the linear combination of $\varphi_{1,2},\varphi_{1,3},\ldots,\varphi_{1,M+2}$.
	Also the linearly independence of $\varphi_{1,2},\varphi_{1,3},\ldots,\varphi_{1,M+2}$ was proved in \cite{Nkaji}.
	Therefore the connection problem for the equation $E_My=0$ is solved by the equation \eqref{connvarM}.
\end{rem}
As a corollary of Theorem \ref{thmconnint}, we show a linear relation for  the Kajihara's $q$-hypergeometric series $W^{M,2}$ \cite{Kaji}.
The Kajihara's $q$-hypergeometric series $W^{M, N}$ is defined as follows:  
\begin{align}
	\notag&W^{M,N}\left(\begin{array}{c}
		\{a_i\}_{1\leq i\leq M}\\
		\{x_i\}_{1\leq i\leq M}
	\end{array}\bigg\lvert \ s; \{u_k\}_{1\leq k\leq N}; \{v_k\}_{1\leq k\leq N}; z\right)\\
	\notag&=\sum_{l\in(\mathbb{Z}_{\geq 0})^M} z^{|l|}\frac{\Delta(xq^l)}{\Delta(x)}\prod_{1\leq i\leq M}\frac{1-q^{|l|+l_i} s x_i/x_M}{1-s x_i/x_M}\\
	\notag&\phantom{=\sum_{l\in(\mathbb{Z}_{\geq 0})^M}}\times \prod_{1\leq j\leq M}\left(\frac{(s x_j/x_M)_{|l|}}{((s q/a_j)x_j/x_M)_{|l|}}\prod_{1\leq i\leq M}\frac{(a_j x_i/x_j)_{l_i}}{(q x_i/x_j)_{l_i}}\right)\\
	&\phantom{=\sum_{l\in(\mathbb{Z}_{\geq 0})^M}}\times \prod_{1\leq k\leq N}\left(\frac{(v_k)_{|l|}}{(s q/u_k)_{|l|}}\prod_{1\leq i\leq M}\frac{(u_k x_i/x_M)_{l_i}}{((s q/v_k)x_i/x_M)_{l_i}}\right),\label{defkaji}
\end{align}
where $|l|=l_1+\cdots +l_M$, $\displaystyle \Delta(x)=\prod_{1\leq i<j\leq M}(x_i-x_j)$, $xq^l=\{x_1q^{l_1},\ldots,x_Mq^{l_M}\}$.
For the details of $W^{M,N}$, see \cite{Kaji,kaji2018,KN}.
\begin{cor}\label{corkajilinear}
	We suppose $a_1\cdots a_{M+3}=q^2 b_1\cdots b_{M+3}$ and put
	\begin{align}
		\notag&W\left(\begin{array}{c}
			\{a_i\}_{1\leq i\leq M+3}\\
			\{b_i\}_{1\leq i\leq M+3}
		\end{array}\right)\\
		\notag&=\frac{(a_{3}/b_{1},b_{3}b_{1}q^2/(a_{2}a_{1}),b_{2}b_{1}q^2/(a_{2}a_{1}),a_{2}q/a_{1},a_{1}/a_{2})_\infty}{a_{1}( b_{3}q/a_{2},b_{2}q/a_{2},b_{1}q/a_{2},b_{3}q/a_{1},b_{2}q/a_{1},b_{1}q/a_{1})_\infty}\prod_{i=4}^{M+3}\frac{(b_i b_{1}q^2/(a_{2}a_{1}),a_i q/a_{2},a_i q/a_{1})_\infty}{(a_i b_{1}q^2/(a_{2}a_{1}),b_i q/a_{2},b_i q/a_{1})_\infty}\\
		&\times W^{M,2}\left(\begin{array}{c}
			\{a_{M+3-i}/b_{M+3-i}\}_{1\leq i\leq M}\\
			\{a_{M+3-i}\}_{1\leq i\leq M}
		\end{array}\bigg\lvert \ \frac{a_4 b_{1}q}{a_{2}a_{1}};\frac{b_{1}q}{a_{2}},\frac{b_{1}q}{a_{1}};\frac{a_4}{b_{3}},\frac{a_4}{b_{2}};\frac{a_{3}}{b_{1}}\right).\label{defW}
	\end{align}
	We have
	\begin{align}
		\label{Wrel3}&\sum_{i=2}^{M+3}D_k W\left(\begin{array}{c}
			a_1,a_{k},a_{2}, \ldots, a_{k-1}, a_{k+1}, \ldots a_{M+3}\\\{b_i\}_{1\leq i\leq M+3}
		\end{array}\right)=0,
	\end{align}
	where 
	\begin{align}
		D_k=\left(\frac{a_k}{a_{1}}\right)^2\prod_{i=1}^{M+3}\frac{\theta(a_k/b_i)}{\theta(a_{1}/b_i)}\prod_{\substack{1\leq i\leq M+3\\i\neq 1}}\theta(a_{1}/a_i)\prod_{\substack{1\leq i\leq M+3\\ i\neq k}}\theta(a_k/a_i)^{-1}.
	\end{align}
\end{cor}
\begin{proof}
A transformation formula between the function $W$ and the Jackson integral of Riemann-Papperitz type  \eqref{integrandJP} was given in \cite{Nkaji} as follows: 
	\begin{align}\label{inttoWpart}
		W\left(\begin{array}{c}
			\{a_i\}_{1\leq i\leq M+3}\\
			\{b_i\}_{1\leq i\leq M+3}
		\end{array}\right)=\frac{-1}{q(1-q)(q)_\infty}\int_{q/a_{1}}^{q/a_{2}}\prod_{i=1}^{M+3}\frac{(a_it)_\infty}{(b_it)_\infty}d_qt.
	\end{align}
Due to this formula and the connection formula \eqref{connvarM}, we have the desired equation \eqref{Wrel3}.
\end{proof}

We consider the case $M=1$ in the following.
As mentioned above, the equation $E_1y=0$ is equivalent to the variant of $q$-hypergeometric equation of degree three \eqref{Int: TakemuraH3}.
So the formula (\ref{connvarM}) gives a connection formula for integral solutions of  (\ref{Int: TakemuraH3}).
\begin{cor}
We put
\begin{align}
	&C_1=\frac{\theta \left(q^{h_1-l_1-\nu +1}, \frac{t_1 q^{h_1-l_2-\nu +1}}{t_2}, 
		\frac{t_1 q^{h_1-l_3-\nu +1}}{t_3}, \frac{xq^{\frac{3}{2}-h_1}}{t_1
			}, \frac{x
			q^{-h_2+\nu -\frac{1}{2}}}{t_2}, \frac{x q^{-h_3+\nu
				-\frac{1}{2}}}{t_3}\right)}{\theta \left(q^{\nu -1}, \frac{t_1
			q^{h_1-h_2}}{t_2}, \frac{t_1 q^{h_1-h_3}}{t_3}, 
			\frac{x q^{\frac{1}{2}-l_1}}{t_1}, \frac{x
			q^{\frac{1}{2}-l_2}}{t_2}, \frac{x
			q^{\frac{1}{2}-l_3}}{t_3}\right)},\\
	&C_2=C_1\bigg|_{h_1\leftrightarrow h_2,\,l_1\leftrightarrow l_2,\,t_1\leftrightarrow t_2},\ 
	C_3=C_1\bigg|_{h_1\leftrightarrow h_3,\,l_1\leftrightarrow l_3,\,t_1\leftrightarrow t_3}.
\end{align}
Then the integrals $\varphi_{1,4}^{\mathcal{H}_3}$, $\varphi_{2,4}^{\mathcal{H}_3}$, $\varphi_{3,4}^{\mathcal{H}_3}$ \eqref{intH3} satisfy the following linear relation:
\begin{align}\label{connvar3}
	C_1 \varphi_{1,4}^{\mathcal{H}_3}+C_2 \varphi_{2,4}^{\mathcal{H}_3}+C_3\varphi_{3,4}^{\mathcal{H}_3}=0.
\end{align}
\end{cor}
\begin{proof}
Due to Theorem \ref{thmconnint} we have 
\begin{align}\label{connvarM=1}
	\sum_{k=2}^{4} \tilde{C}_k \int_{q/(Ax)}^{q/a_k}\frac{(Axt)_{\infty}}{(Bxt)_{\infty}} \prod_{i=2}^{4}\frac{(a_i t)_\infty}{(b_i t)_\infty} d_q t=0.
\end{align}
Here 
\begin{align}
	\tilde{C}_2 
	&=\left(\frac{a_2}{Ax}\right)^2 \frac{\theta(a_2/Bx, a_2/b_2, a_2/b_3, a_2/b_4, Ax/a_2, Ax/a_3, Ax/a_4)}{\theta(A/B, Ax/b_2, Ax/b_3, Ax/b_4, a_2/Ax, a_2/a_3, a_2/a_4)} \nonumber\\
	&=- \frac{\theta( a_2/b_2, a_2/b_3, a_2/b_4,q^2Bx/a_2, Ax/a_3, Ax/a_4)}{\theta(q^2B/A, a_2/a_3, a_2/a_4,Ax/b_2, Ax/b_3, Ax/b_4)},\\
	\tilde{C}_3 &= \tilde{C}_2\bigg|_{a_2\leftrightarrow a_3,\,b_2\leftrightarrow b_3},\ \tilde{C}_4 = \tilde{C}_2\bigg|_{a_2\leftrightarrow a_4,\,b_2\leftrightarrow b_4}.
\end{align}
By putting $ q^{-\nu}=B/A$, $q^{l_{i}-1/2}t_{i}=b_{i+1}/A$, $q^{h_{i}+1/2}t_{i}=a_{i+1}/B$, we have the desired relation \eqref{connvar3}.
For the reason why we put parameters so, see Remark \ref{remparameterstrans}.
\end{proof}
\begin{rem}
As mentioned above, the integrals $\varphi_{i,j}^{\mathcal{H}_3}$ satisfy the trivial relation:
\begin{align}
	\varphi_{i,j}^{\mathcal{H}_3}+\varphi_{j,k}^{\mathcal{H}_3}+\varphi_{k,i}^{\mathcal{H}_3}=0.
\end{align}
Therefore we get
\begin{align}
	\begin{pmatrix}
		1&-1&0&1&0&0\\
		1&0&-1&0&1&0\\
		0&1&-1&0&0&1\\
		0&0&C_1&0&C_2&C_3
	\end{pmatrix}
	\begin{pmatrix}
		\varphi_{1,2}^{\mathcal{H}_3}\\
		\varphi_{1,3}^{\mathcal{H}_3}\\
		\varphi_{1,4}^{\mathcal{H}_3}\\
		\varphi_{2,3}^{\mathcal{H}_3}\\
		\varphi_{2,4}^{\mathcal{H}_3}\\
		\varphi_{3,4}^{\mathcal{H}_3}
	\end{pmatrix}
	=
	\begin{pmatrix}
		0\\0\\0\\0
	\end{pmatrix}.
\end{align}
The rank of this $4\times6$ matrix is $4$ in general.
\end{rem}
\begin{rem}
We obtained a three term relation (\ref{connvarM=1}) for the Jackson integral of Riemann-Papperitz type.
On the other hand, the integral can be transformed to the very-well-poised $q$-hypergeometric function ${}_8W_7$ by the Bailey's formula \cite[(2.10.18)]{GR}:
\begin{align}
	\notag&\int_{a}^{b}\frac{(qt/a,qt/b,ct,dt)_{\infty}}{(et,ft,gt,ht)_{\infty}}d_{q}t\\
	=&b(1-q)\frac{(q,bq/a,a/b,cd/(eh),cd/(fh),cd/(gh),bc,bd)_{\infty}}{(ae,af,ag,be,bf,bg,bh,bcd/h)_{\infty}}
	\times {}_{8}W_{7}\left(\frac{bcd}{hq};be,bf,bg,\frac{c}{h},\frac{d}{h};ah\right),\label{inttoser}
\end{align}
where $cd=abefgh$ and
\begin{align}
	{}_{r+1}W_{r}(a_{1}; a_{4},\ldots, a_{r+1};z)=\sum_{n=0}^{\infty}\frac{1-a_{1}q^{2n}}{1-a_{1}}\frac{(a_{1},a_{4},\ldots,a_{r+1})_{n}}{(q,qa_{1}/a_{4},\ldots,qa_{1}/a_{r+1})_{n}}z^{n}.
\end{align}
So the relation \eqref{connvarM=1} can be written in terms of ${}_8W_7$ as follows:
\begin{align}\label{w87_conec}
	\nonumber&{}_8W_7\left(a; b, c, d, e, 
	f; \frac{q^2a^2}{bcdef}\right)\\
	\nonumber&= \frac{a}{b}\frac{(q a, c, b/a, c/a)_{\infty}}{(a/b, c/b, bc/a, b c/a, q b^2/a)_{\infty}}\prod_{\eta \in \{ d, e, f \}} \frac{(q a/b \eta, q b/\eta)_{\infty}}{(q/\eta, q a/\eta)_{\infty}}
	{}_8W_7\left(\frac{b^2}{a}; b,\frac{b e}{a}, \frac{b c}{a}, \frac{b d}{a},  \frac{b f}{a}; \frac{q^2a^2}{bcdef}\right) \\
	&+ \mathrm{idem}(b,c).
\end{align} 
Here, the symbol $\mathrm{idem}(b,c)$ after an expression means that the preceding expression is repeated with $b$ and $c$ interchanged.
For the relation to known connection formula, see Appendix \ref{Appendix1}.
\end{rem}

\section{Summary}\label{Secsumdis}

In this paper, we gave a connection formula for the Jackson integral of Riemann-Papperitz type \eqref{intRP} in Theorem \ref{thmconnint}.
This formula can be obtained by reducing Mimachi's connection formula for the Jackson integral of Jordan-Pochhammer type suitably.
We solved the connection problem for the equation $E_My=0$ \eqref{defequationEM}.
When $M=1$, this equation is equivalent to the variant of $q$-hypergeometric equation of degree three \cite{HMST}. 
It follows that the connection problem for this variant is also solved.
Due to \eqref{inttoWpart}, the Jackson integral of Riemann-Papperitz type is expressed by the Kajihara's $q$-hypergeometric series $W^{M,2}$.
So we got a linear relation for $W^{M,2}$ (Corollary \ref{corkajilinear}).

There are many related problems.
\begin{itemize}
\item
The $q$-hypergeometric function ${}_8W_7$  possesses the $W(D_5)$-symmetry \cite{LV}.
Also, the symmetry, which includes $W(D_5)$, of the three term relations for ${}_8W_7$ was obtained by \cite{LV}.
These symmetries are helpful in determining whether given two  transformation formulas for ${}_8W_7$ are equivalent or not.
The integral representation \eqref{inttoWpart}  gives two and three term relations for $W^{M, 2}$ \cite[(3.17), (3.18)]{Nkaji}.
In addition, we got another linear relation \eqref{Wrel3} for $W^{M, 2}$.
It will be meaningful to clarify the symmetry of the relations for $W^{M, 2}$.

\item
We gave a connection formula for the Jackson integral of Riemann-Papperitz type \eqref{intRP}.
An elliptic extension of the integral \eqref{intRP} will be interesting, see \cite{R, S}.
The elliptic extension $E^{M, N}$ of the Kajihara's $q$-hypergeometric function $W^{M, N}$ has been studied by \cite{KN}.
It is important to derive a linear equation for the function $E^{M, N}$ and its connection formulas.
\item
The $6j$-symbols arise from the context of representations of classical or quantum groups \cite{R6j}. 
The $q$-Racah polynomials appear as the $6j$-symbols of the quantum group $U_q(\mbox{sl}(2))$ \cite{KR}.
The $q$-Racah polynomial can be represented by the function ${}_8W_7$.
The series $W^{M, 2}$ is an extension of ${}_8W_7$.
It is interesting to derive $W^{M, 2}$ from the viewpoint of representations of some quantum group.
\end{itemize}


\appendix\section{Relation to Gupta-Masson's formula}\label{Appendix1}
In the main text, we obtained a three term relation \eqref{w87_conec} for the function ${}_8W_7$ as an application of the connection formula \eqref{Int: connvarM} for the equation $E_1y=0$. 
In this appendix, we describe the relation between our formula \eqref{w87_conec} and a formula which can be derived by known method \cite{GM}.

In \cite{GM}, the following equation and its twelve solutions in term of the function ${}_8W_7$ is discussed:
\begin{align}\label{GMdiff}
&Y_{n + 1} - c_n Y_{n} + d_n Y_{n - 1}=0,
\\
\nonumber&c_n = C_n + D_n + \frac{a^2 q^{2n + 2}}{bcdeh^2}\frac{(1-b)(1-c)(1-d)(1-e)}{(1 - aq^{n + 1}/h)}, 
\\
\nonumber&d_n = C_{n - 1}D_n, 
\\
\nonumber&C_n = - \frac{(1 - aq^{n+1}/bh)(1-aq^{n+1}/ch)(1-aq^{n+1}/dh)(1-aq^{n+1}/eh)}{(1-aq^{n+1}/h)},
\\
\nonumber&D_n=-q(1-q^n/h)(1-aq^n/h(1-a^2q^{n+1}/bcdeh)).
\end{align}
The equation \eqref{GMdiff} is equivalent to the equation $E_1y=0$ \eqref{defequationEM}.
More precisely, by putting $a = a_1/B, b =q b_1/A, c = q b_2/A, d = q b_3/A, e=a_2/B, hq^{-n}=a/x$, and by the gauge transformation $Y_n = (-1)^n y(x)/(qx/b, qx/c, qx/d)_{\infty}$, the equation \eqref{GMdiff} is transformed to $E_1y=0$.
It is also claimed in \cite{GM} that the connection formulas among any three of twelve solutions can be calculated in principle by using the well-known relations \cite[(I\hspace{-.1em}I\hspace{-.1em}I.23), (I\hspace{-.1em}I\hspace{-.1em}I.37)]{GR}.
So our relation \eqref{w87_conec} should be also derived by the method in \cite{GM}.
In the following, we compare this relation with our formula  \eqref{w87_conec}.

Gupta and Masson \cite{GM} give the following equation: 
\begin{align}\label{w87_conecGM}
\nonumber&{}_8W_7\left(a; b, c, d, e, f; \frac{q^2a^2}{bcdef}\right)\\
\nonumber&=
\frac{\theta \left(b/c \right)}{\theta \left(cdef/a^2 \right)}\frac{(qa, b, qa/cd, qa/ce, qa/de, qa/df, qa/ef)_{\infty}}{(q^2/a)_{\infty}}
\prod_{\mu \in \{ b, c, d, e, f \}}\frac{(q\eta/a)_{\infty}}{(q/\eta, qa/\eta, b\eta/a)_{\infty}}\\
\nonumber&\times {}_8W_7\left(\frac{q}{a}; \frac{q}{b},  \frac{q}{c},  \frac{q}{d},  \frac{q}{e}, \frac{q}{f} \right) +
\frac{(aq, c, c/a, qb/a)_{\infty}}{(c/b, qa/b, bc/a, qb^2/a)_{\infty}}
\prod_{\mu \in \{ d, e, f \}}\frac{(qb/\eta, qa/b\eta)_{\infty}}{(q/\eta, qa/\eta)_{\infty}}\\
&
\times
\left(
1 - \frac{\theta(b, a/b, cd/a, ce/a, a/cf, bedf/a^2)}{\theta(c, a/c, bd/a, be/a, a/bf, cdef/a^2)}
\right)
{}_8W_7\left(\frac{b^2}{a}; b,\frac{b e}{a}, \frac{b c}{a}, \frac{b d}{a},  \frac{b f}{a}; \frac{q^2a^2}{bcdef}\right).
\end{align}
Subtracting the equation \eqref{w87_conecGM} from the equation in which the parameters $b$ and $c$ are interchanged, we get the following equation:
\begin{align}\label{w87 coneckiti}
&{}_8W_7\left(a; b, c, d, e, f; \frac{q^2a^2}{bcdef}\right)=C(b, c) {}_8W_7\left(\frac{b^2}{a}; b,\frac{b e}{a}, \frac{b c}{a}, \frac{b d}{a},  \frac{b f}{a}; \frac{q^2a^2}{bcdef}\right) + \mathrm{idem}(b,c), 
\end{align}
where
\begin{align}\label{coefficient connec}
&C(b, c) =\frac{D(b, c)}{E(b, c)} \frac{A(b, c) - A(c, b)}{B(b, c) - B(c, b)}, \\
\nonumber&A(x, y) = \theta \left(x, \frac{a}{x}, \frac{a}{y f}, \frac{yd}{a}, \frac{ye}{a}, \frac{xdef}{a^2} \right), \\
\nonumber&B(x, y) = \theta  \left( x, \frac{x}{a}, \frac{xdef}{a^2} \right)\prod_{\eta \in \{d, e, f \}} \theta \left( \frac{y \eta}{a} \right), \\
\nonumber&D(b, c) = \theta \left(  \frac{a}{cf} \right)  \left( \frac{c}{a}, \frac{qb}{a}, \frac{qa}{c}, \frac{qa}{bf} \right)_{\infty}^2  \left(c, \frac{a}{q}, \frac{a}{b}, \nonumber\frac{b}{c}, \frac{bf}{a}, \frac{qa}{bd}, \frac{qa}{be}, \frac{qc^2}{a}\right)_{\infty} \prod_{\eta \in \{ d, e, f \}}\left( \frac{qc}{\eta} \right)_{\infty}, \\
\nonumber&E(b, c)= \left( \frac{b}{c}, \frac{c}{b} \right)_{\infty} \prod_{\eta \in \{ b, c \} }\left( \frac{q\eta^2}{a} \right)_{\infty} \theta \left( \frac{a}{\eta} \right)\prod_{\tau \in \{ d, e, f \} }\left( \frac{q}{\tau} \right)_{\infty}\prod_{\mu \in \{ b, c, d, e, f \} }\left( \frac{qa}{\mu} \right)_{\infty}.
\end{align} 
This equation \eqref{w87 coneckiti} is the there term relation for ${}_8W_7$ which should be compared with formula \eqref{w87_conec}.
It is interesting that the coefficients of the formula  \eqref{w87_conec} are much simpler than \eqref{coefficient connec}.

\section*{Acknowledgements}
The authors would like to thank Yasuhiko Yamada for valuable suggestions, helpful discussions and constant encouragements.
This work is supported by JST SPRING, Grant Number JPMJSP2148 and JSPS KAKENHI Grant Number 22H01116.

\end{document}